\documentclass{amsart}

\usepackage[utf8]{inputenc} 
\usepackage[T1]{fontenc} 
\usepackage[english]{babel}

\theoremstyle{plain} 
\newtheorem{theorem}{Theorem} 
\newtheorem{lemma}[theorem]{Lemma} 

\theoremstyle{remark} 
\newtheorem*{remark}{Remark} 
\newtheorem*{question}{Question}

\begin{document} 

\title[Linear functions and duality on the infinite polytorus]{Linear functions and duality \\
on the infinite polytorus} 
\date{\today} 

\author{Ole Fredrik Brevig} 
\address{Department of Mathematical Sciences, Norwegian University of Science and Technology (NTNU), NO-7491 Trondheim, Norway} 
\email{ole.brevig@math.ntnu.no}

\begin{abstract}
	We consider the following question: Are there exponents $2<p<q$ such that the Riesz projection is bounded from $L^q$ to $L^p$ on the infinite polytorus? We are unable to answer the question, but our counter-example improves a result of Marzo and Seip by demonstrating that the Riesz projection is unbounded from $L^\infty$ to $L^p$ if $p\geq 3.31138$. A similar result can be extracted for any $q>2$. Our approach is based on duality arguments and a detailed study of linear functions. Some related results are also presented.
\end{abstract}


\maketitle

\section{Introduction} \label{sec:intro} 
Let $\mathbb{T}^\infty = \mathbb{T}\times\mathbb{T}\times\cdots$ denote the countably infinite cartesian product of the torus $\mathbb{T} = \{z \in \mathbb{C} \,:\, |z|=1 \}$. We equip the $\mathbb{T}^\infty$ with its Haar measure $\mu_\infty$, which is equal to the infinite product of the normalized Lebesgue arc measure on $\mathbb{T}$ in each variable. Let $1\leq p \leq \infty$. Every $f$ in $L^p(\mathbb{T}^\infty)$ has a Fourier series expansion
\[f(z) = \sum_{\alpha \in \mathbb{Z}_0^{\infty}} c_\alpha z^{\alpha}\]
where the Fourier coefficients are defined in the standard way and $\alpha \in \mathbb{Z}_0^\infty$ means that the multi-index $\alpha$ contains only a finite number of non-zero components. The Riesz projection on $\mathbb{T}^\infty$ is defined by 
\begin{equation}\label{eq:rieszprojection} 
	Pf(z) = \sum_{\alpha \in \mathbb{N}_0^\infty} c_\alpha z^{\alpha}. 
\end{equation}
The initial motivation for the present paper is the following. 
\begin{question}
	What is the largest $p=p_\infty$ such that the Riesz projection \eqref{eq:rieszprojection} is bounded from $L^\infty(\mathbb{T}^\infty)$ to $L^p(\mathbb{T}^\infty)$? 
\end{question}

The Riesz projection is certainly a contraction on the Hilbert space $L^2(\mathbb{T}^\infty)$ and since $\|f\|_{L^2(\mathbb{T}^\infty)} \leq \|f\|_{L^\infty(\mathbb{T}^\infty)}$, we get that $p_\infty\geq2$. This question has previously been investigated by Marzo and Seip \cite{MS11} who demonstrated that $p_\infty \leq 3.67632$. We will obtain the following improvement. 
\begin{theorem}\label{thm:331} 
	$p_\infty\leq p = 3.31138\ldots$, where $p$ denotes the unique positive solution of the equation
	\[{\Gamma\left(1+\frac{p}{2}\right)}^\frac{1}{p} = \frac{2}{\sqrt{\pi}}.\]
\end{theorem}

For $2\leq p \leq q \leq \infty$, let $\|P\|_{q,p}$ denote the norm of the Riesz projection from $L^q(\mathbb{T}^\infty)$ to $L^p(\mathbb{T}^\infty)$. In the case that the Riesz projection is unbounded, we use the convention $\|P\|_{q,p}=\infty$. As explained in \cite{MS11}, for each fixed $2\leq q\leq \infty$ there is a number $2 \leq p_q \leq q$, called the critical exponent, with the property that 
\begin{equation}\label{eq:dichotomy} 
	\|P\|_{p,q} = 
	\begin{cases}
		1 & \text{if } p \leq p_q, \\
		\infty & \text{if } p>p_q. 
	\end{cases}
\end{equation}
The dichotomy \eqref{eq:dichotomy} is a direct consequence of the fact that we are on the infinite polytorus. Let $f$ be a function in the unit ball of $L^q(\mathbb{T}^\infty)$ such that $\|Pf\|_{L^p(\mathbb{T}^\infty)}>1$. Consider the function
\[f_2(z) = f(z_1,z_3,z_5,\ldots)\cdot f(z_2,z_4,z_6,\ldots)\]
which is also in the unit ball of $L^q(\mathbb{T}^\infty)$. The Riesz projection \eqref{eq:rieszprojection} acts independently on the variables, so we find that
\[Pf_2(z) = Pf(z_1,z_3,z_5,\ldots) \cdot Pf(z_2,z_4,z_6,\ldots)\]
which implies that $\|Pf_2\|_{L^p(\mathbb{T}^\infty)} = \|Pf\|_{L^p(\mathbb{T}^\infty)}^2 > \|P f\|_{L^p(\mathbb{T}^\infty)}$. This procedure can be repeated and so we obtain \eqref{eq:dichotomy}. The example from \cite{MS11} producing $p_\infty \leq 3.67632$ is a function of only two variables.

The present paper is inspired by \cite{BP15}, where linear functions are used as building blocks in an similar way to what was just described to construct a counter-example related to Nehari's theorem for Hankel forms on $\mathbb{T}^\infty$. The example from \cite{BP15} improves on an earlier example from \cite{OCS12} by replacing a function of two variables by a linear function in an infinite number of variables.

Our approach differs from that of \cite{MS11} (and \cite{BOCSZ}) in that we do not attempt to directly construct a counter-example, but instead use duality arguments to infer its existence. This approach leads us to consider the Hardy spaces $H^p(\mathbb{T}^\infty)$, which are the subspaces of $L^p(\mathbb{T}^\infty)$ consisting of elements such that $Pf = f$. A standard argument involving the Hahn--Banach theorem (see e.g. \cite[Sec.~7.2]{Duren}) yields that 
\begin{equation}\label{eq:hahnbanach} 
	\inf_{P \psi = \varphi} \|\psi\|_{L^q(\mathbb{T}^\infty)} = \|\varphi\|_{(H^r(\mathbb{T}^\infty))^\ast} = \sup_{f \in H^r(\mathbb{T}^\infty)} \frac{|\langle f, \varphi \rangle_{L^2(\mathbb{T}^\infty)}|}{\|f\|_{H^r(\mathbb{T}^\infty)}} 
\end{equation}
for $1\leq r < \infty$ and $q^{-1}+r^{-1}=1$. We will choose $\varphi$ and try to find the optimal $f$ in $H^r(\mathbb{T}^\infty)$ attaining the supremum. This will ensure the existence of $\psi$ in $L^q(\mathbb{T}^\infty)$ attaining the infimum, which be our counter-example through \eqref{eq:dichotomy}. 

We shall see in Section~\ref{sec:minimal} that if we know the optimal $f$ in the supremum on the right hand side of \eqref{eq:hahnbanach}, we can use H\"older's inequality to construct the element $\psi$ in $L^q(\mathbb{T}^\infty)$ of minimal norm such that $P\psi=\varphi$, thereby attaining the infimum on the left hand side of \eqref{eq:hahnbanach}.

As in \cite{BP15} we will primarily be working with linear functions, which are of the form 
\begin{equation}\label{eq:linearfunction} 
	f(z) = \sum_{j=1}^\infty c_j z_j. 
\end{equation}
Clearly, $\|f\|_{H^2(\mathbb{T}^\infty)}^2 = \sum_{j\geq1}|c_j|^2$ and we easily check that $\|f\|_{H^\infty(\mathbb{T}^\infty)} = \sum_{j\geq1} |c_j|$. For $1\leq p< \infty$, optimal norm estimates are given by Khintchine's inequality. Define
\[a_p = \min\left(1,\,{\Gamma\left(1+\frac{p}{2}\right)}^\frac{1}{p}\right) \qquad \text{and} \qquad b_p = \max\left(1,\,{\Gamma\left(1+\frac{p}{2}\right)}^\frac{1}{p}\right).\]
If $f$ is a linear function \eqref{eq:linearfunction} and $1\leq p < \infty$, then we restate a result from \cite{KK01} as 
\begin{equation}\label{eq:khintchine} 
	a_p \|f\|_{H^2(\mathbb{T}^\infty)} \leq \|f\|_{H^p(\mathbb{T}^\infty)} \leq b_p \|f\|_{H^2(\mathbb{T}^\infty)} 
\end{equation}
and the constants in \eqref{eq:khintchine} are optimal. We shall obtain the following companion inequality for dual norms, which might be of independent interest. 
\begin{theorem}\label{thm:khintchinedual} 
	Let $1\leq p < \infty$. If $f$ is a linear function \eqref{eq:linearfunction}, then 
	\begin{equation}\label{eq:khintchinedual} 
		b_p^{-1} \|f\|_{H^2(\mathbb{T}^\infty)} \leq \|f\|_{(H^p(\mathbb{T}^\infty))^\ast} \leq a_p^{-1} \|f\|_{H^2(\mathbb{T}^\infty)}. 
	\end{equation}
	The constants are optimal. 
\end{theorem}
\begin{remark}
	In the case $p=\infty$, it is easy to deduce by similar considerations (Lemma~\ref{lem:projection}) that $\|f\|_{(H^\infty(\mathbb{T}^\infty))^\ast} = \sup_{j\geq1} |c_j|$ if $f$ is a linear function \eqref{eq:linearfunction}. 
\end{remark}

Optimality of the constants containing the Gamma function in \eqref{eq:khintchine} and \eqref{eq:khintchinedual} both arise from the function
\[f(z) = \frac{z_1+z_2+\cdots + z_d}{\sqrt{d}}\]
as $d\to\infty$ through the central limit theorem. In view of \eqref{eq:dichotomy} and \eqref{eq:hahnbanach}, we can therefore obtain the following general result. Note that Theorem~\ref{thm:331} corresponds to the particular case $q=\infty$, since $\Gamma(3/2) = \sqrt{\pi}/2$. 
\begin{theorem}\label{thm:general} 
	Let $2 \leq p \leq q \leq \infty$ and set $q^{-1}+r^{-1}=1$. If
	\[{\Gamma\left(1+\frac{p}{2}\right)}^\frac{1}{p} {\Gamma\left(1+\frac{r}{2}\right)}^\frac{1}{r}>1,\]
	then the Riesz projection is unbounded from $L^q(\mathbb{T}^\infty)$ to $L^p(\mathbb{T}^\infty)$. 
\end{theorem}

\begin{remark}
	Theorem~\ref{thm:general} is an improvement on the same statement with requirement $p/2 \cdot r/2 >1$, which can be deduced from a one-variable example found in \cite[Sec.~4]{BOCSZ}. Here is an alternative example to that of \cite{BOCSZ} obtained by our approach using the Hahn--Banach theorem. For $w \in \mathbb{D}$, the functional of point evaluation $f \mapsto f(w)$ has norm $(1-|w|^2)^{-1/r}$ on $H^r(\mathbb{T})$ and the analytic symbol is $\varphi_w(z) = (1-\overline{w}z)^{-1}$. Hence, if $w=\varepsilon>0$ then $\|\varphi_\varepsilon\|_{(H^r(\mathbb{T}))^\ast} = 1 + r^{-1} \varepsilon^2 + O(\varepsilon^4)$ as $\varepsilon\to0$. Furthermore, 
	\[\|\varphi_\varepsilon\|_{H^p(\mathbb{T})} = \big\|(1-\varepsilon z)^{-p/2}\big\|_{H^2(\mathbb{T})}^{2/p} = 1 + \frac{p}{4} \varepsilon^2 + O(\varepsilon^4),\]
	so we obtain the desired counter-example as soon as $r^{-1}>p/4$ in view of \eqref{eq:dichotomy}. The optimal $\psi_w$ in $L^q(\mathbb{T})$ for this functional can be found in \cite[Thm.~6.1]{CG86}, and we note that it is similar (but not equal to) the counter-example constructed in \cite{BOCSZ}.
\end{remark}

The present paper is organised into two additional sections. In Section~\ref{sec:linear} we prove Theorem~\ref{thm:khintchinedual} and Theorem~\ref{thm:general}. Section~\ref{sec:minimal} is devoted to constructing the element $\psi$ in $L^q(\mathbb{T}^\infty)$ for $1<q\leq \infty$ of minimal norm such that $P \psi(z) = z_1+z_2+\cdots+z_d$, thereby realising the infimum \eqref{eq:hahnbanach} in this special case, which is of particular interest due to the crucial role it plays in the proof of Theorem~\ref{thm:khintchinedual} and Theorem~\ref{thm:general}.

\section{Linear functions on $\mathbb{T}^\infty$} \label{sec:linear} 
In preparation for the proof of Theorem~\ref{thm:khintchinedual} and Theorem~\ref{thm:general}, let us recall some basic facts about linear functions and projections on $\mathbb{T}^\infty$. The projection $A_d$ obtained by formally setting $z_j = 0$ for $j>d$ has the representation
\[A_d f(z_1,z_2,\ldots) = \int_{\mathbb{T}^\infty} f(z_1,z_2,\ldots,z_d,z_{d+1},z_{d+2},\ldots) \,d\mu_\infty(z_{d+1},z_{d+2},\ldots).\]
Since $A_d f$ is a function the first $d$ variables, we take $L^p$ norm with respect to these variables and use the triangle inequality to obtain 
\begin{equation}\label{eq:Ad} 
	\|A_d f\|_{L^p(\mathbb{T}^\infty)} \leq \|f\|_{L^p(\mathbb{T}^\infty)}. 
\end{equation}
Let $k\in\mathbb{Z}$. We say that $f$ is $k$-homogeneous if
\[f(e^{i\theta}z_1,e^{i\theta} z_2,e^{i\theta}z_3,\ldots) = e^{ki\theta} f(z_1,z_2,z_3,\ldots).\]
Clearly every $f$ in $L^p(\mathbb{T}^\infty)$ can be decomposed in $k$-homogeneous parts, say 
\begin{equation}\label{eq:homocomp} 
	f(z) = \sum_{k\in\mathbb{Z}} f_k(z), 
\end{equation}
where $f_k$ is $k$-homogeneous. The following simple lemma is well-known, but we include a short proof for the readers convenience. 
\begin{lemma}\label{lem:projection} 
	Let $1\leq p \leq \infty$ and suppose that $f$ in $L^p(\mathbb{T}^\infty)$ is decomposed as in \eqref{eq:homocomp}. Then $\|f_k\|_{L^p(\mathbb{T}^\infty)} \leq \|f\|_{L^p(\mathbb{T}^\infty)}$ for every $k\in\mathbb{Z}$. 
\end{lemma}
\begin{proof}
	By the decomposition \eqref{eq:homocomp}, we find that
	\[f_k(z) = \int_{-\pi}^\pi f(e^{i\theta}z_1,e^{i\theta} z_2,e^{i\theta}z_3,\ldots) \,e^{-ki\theta}\,\frac{d\theta}{2\pi}.\]
	By the triangle inequality and interchanging the order of integration, we obtain
	\[\|f_k\|_{L^p(\mathbb{T}^\infty)}^p \leq \int_{-\pi}^\pi \int_{\mathbb{T}^\infty} \left|f(e^{i\theta}z_1,e^{i\theta} z_2,e^{i\theta}z_3,\ldots)\right|^p \,d\mu_\infty(z)\, \frac{d\theta}{2\pi} = \|f\|_{L^p(\mathbb{T}^\infty)}^p,\]
	since for each $\theta$ the rotation $z_j \mapsto e^{i\theta}z_j$ does not change the $L^p(\mathbb{T}^\infty)$ norm of $f$. 
\end{proof}

Let $\operatorname{Lin}(\mathbb{T}^\infty)$ denote the space of linear functions \eqref{eq:linearfunction}. Lemma~\ref{lem:projection} states that the projection from $H^p(\mathbb{T}^\infty)$ to $\operatorname{Lin}(\mathbb{T}^\infty) \cap H^p(\mathbb{T}^\infty)$ is contractive. This fact is crucial to the proof of Theorem~\ref{thm:khintchinedual} and Theorem~\ref{thm:general} since it allows us to compute the $(H^p(\mathbb{T}^\infty))^\ast$ norm of a linear function $\varphi$ by testing only against functions $f$ from $\operatorname{Lin}(\mathbb{T}^\infty) \cap H^p(\mathbb{T}^\infty)$. 

In view of Khintchine's inequality \eqref{eq:khintchine}, the space $\operatorname{Lin}(\mathbb{T}^\infty) \cap H^p(\mathbb{T}^\infty)$ consists of linear functions \eqref{eq:linearfunction} with square summable coefficients for each $1\leq p < \infty$, although the norms are generally different. 

Armed with these preliminaries, we will now obtain the key new ingredient needed in the proofs of Theorem~\ref{thm:khintchinedual} and Theorem~\ref{thm:general}.
\begin{lemma}\label{lem:dualinverse} 
	Let $1\leq p< \infty$ and set $\varphi_d(z) = (z_1+\cdots+z_d)/\sqrt{d}$. Then
	\[\|\varphi_d\|_{(H^p(\mathbb{T}^\infty))^\ast} = \|\varphi_d\|_{H^p(\mathbb{T}^\infty)}^{-1}.\]
\end{lemma}
\begin{proof}
	For the lower bound, we simply note that since $\varphi_d$ is in $H^p(\mathbb{T}^\infty)$ we obtain
	\begin{equation}\label{eq:lowerbound} 
		\|\varphi_d\|_{(H^p(\mathbb{T}^\infty))^\ast} = \sup_{f \in H^p(\mathbb{T}^\infty)} \frac{|\langle f, \varphi_d \rangle_{H^2(\mathbb{T}^\infty)}|}{{\|f\|_{H^p(\mathbb{T}^\infty)}}} \geq \frac{\|\varphi_d\|_{H^2(\mathbb{T}^\infty)}^2}{\|\varphi_d\|_{H^p(\mathbb{T}^\infty)}} = \|\varphi_d\|_{H^p(\mathbb{T}^\infty)}^{-1}. 
	\end{equation}
	For the upper bound, we first use \eqref{eq:Ad} and Lemma~\ref{lem:projection} to the effect that 
	\begin{equation}\label{eq:upperdual} 
		\|\varphi_d\|_{(H^p(\mathbb{T}^\infty))^\ast} = \sup_{f \in H^p(\mathbb{T}^\infty)} \frac{|\langle f, \varphi_d \rangle_{H^2(\mathbb{T}^\infty)}|}{{\|f\|_{H^p(\mathbb{T}^\infty)}}} = \sup_{f \in \operatorname{Lin}(\mathbb{T}^d)} \frac{|\langle f, \varphi_d \rangle_{H^2(\mathbb{T}^\infty)}|}{{\|f\|_{H^p(\mathbb{T}^\infty)}}}. 
	\end{equation}
	Any non-trivial element $f$ in $\operatorname{Lin}(\mathbb{T}^d)$ is of the form
	\[f(z) = \sum_{j=1}^d c_j z_j\]
	with at least one non-zero coefficient. Define 
	\begin{equation}\label{eq:lambda} 
		\lambda = \langle f, \varphi_d\rangle_{H^2(\mathbb{T}^\infty)} = \frac{c_1+\cdots+c_d}{\sqrt{d}}.
	\end{equation}
	After rotating each of the variables if necessary, we may assume that $c_j\geq0$ for $1\leq j \leq d$ so that $\lambda>0$ whenever $f$ is a non-trivial element in $\operatorname{Lin}(\mathbb{T}^d)$.
	
	For $1\leq k \leq d$, let $f_k$ denote the polynomial obtained by replacing the coefficient sequence $(c_1,\ldots, c_d)$ of $f$ with the shifted sequence
	\[(c_k,c_{k+1},\ldots,c_d,c_1,\ldots,c_{k-1}).\]
	By symmetry, we find that $\|f_k\|_{H^p(\mathbb{T}^\infty)} = \|f\|_{H^p(\mathbb{T}^\infty)}$. Note also that
	\[\frac{1}{d}\sum_{k=1}^d f_k(z) = \frac{c_1+\cdots+c_d}{d} \sum_{j=1}^d z_j = \lambda \varphi_d(z).\]
	The triangle inequality therefore allows us to conclude that 
	\begin{equation}\label{eq:triangle} 
		\lambda\|\varphi_d\|_{H^p(\mathbb{T}^\infty)} \leq \frac{1}{d} \sum_{k=1}^d \|f_k\|_{H^p(\mathbb{T}^\infty)} = \|f\|_{H^p(\mathbb{T}^\infty)}. 
	\end{equation}
	Using \eqref{eq:upperdual} with \eqref{eq:lambda} and \eqref{eq:triangle}, we obtain the upper bound
	\[\|\varphi_d\|_{(H^p(\mathbb{T}^\infty))^\ast} = \sup_{f \in H^p(\mathbb{T}^\infty)} \frac{|\langle f, \varphi_d \rangle_{H^2(\mathbb{T}^\infty)}|}{{\|f\|_{H^p(\mathbb{T}^\infty)}}} \leq \frac{\lambda}{\lambda \|\varphi_{d}\|_{H^p(\mathbb{T}^\infty)}} = \|\varphi_{d}\|_{H^p(\mathbb{T}^\infty)}^{-1}\]
	which, when combined with the lower bound \eqref{eq:lowerbound}, completes the proof. 
\end{proof}

Another viewpoint is to consider $(z_j)_{j\geq1}$ a sequence of independently distributed random variables on the torus and $f(z) = \sum_{j\geq1} c_j z_j$ as a weighted random walk in the plane. The norms $\|f\|_{H^p(\mathbb{T}^\infty)}$ can now be interpreted as moments of this random walk. A simple computation (see Section~\ref{sec:minimal}) gives that $\|z_1+z_2\|_{H^1(\mathbb{T}^\infty)} = 4/\pi$ and it is demonstrated in \cite{BNSW11} that
\[\|z_1+z_2+z_3\|_{H^1(\mathbb{T}^\infty)} = \frac{3}{16}\frac{2^{1/3}}{\pi^4} \Gamma^6\left(\frac{1}{3}\right)+\frac{27}{4}\frac{2^{2/3}}{\pi^4} \Gamma^6\left(\frac{2}{3}\right) = 1.57459\ldots\]
In general it is difficult to compute $\|f\|_{H^p(\mathbb{T}^\infty)}$ even for simple linear polynomials $f$ (when $p$ is not an even integer). However, the central limit theorem gives that 
\begin{equation}\label{eq:clt} 
	\lim_{d\to\infty}\left\|\frac{z_1+z_2+\cdots+z_d}{\sqrt{d}}\right\|_{H^p(\mathbb{T}^\infty)}^p = \int_{\mathbb{C}} |Z|^p e^{-|Z|^2}\, \frac{dZ}{\pi} = \Gamma\left(1+\frac{p}{2}\right), 
\end{equation}
since $(z_1+z_2+\cdots+z_d)/\sqrt{d}$ has a limiting complex normal distribution.

We are now ready to prove Theorem~\ref{thm:khintchinedual}. To conform with the notations of the present section and to make the proof clearer, we consider now $\varphi$ in $(H^p(\mathbb{T}^\infty))^\ast$ and $f$ in $H^p(\mathbb{T}^\infty)$, so $\varphi$ plays the role of $f$ in the statement of the theorem. 
\begin{proof}
	[Proof of Theorem~\ref{thm:khintchinedual}] Let $\varphi$ be a linear function in $(H^p(\mathbb{T}^\infty))^\ast$. By Lemma~\ref{lem:projection}, the Cauchy--Schwarz inequality and Khintchine's inequality \eqref{eq:khintchine}, we find that
	\begin{multline*}
		\|\varphi\|_{(H^p(\mathbb{T}^\infty))^\ast} = \sup_{f \in \operatorname{Lin}(\mathbb{T}^\infty)} \frac{|\langle f, \varphi \rangle_{H^2(\mathbb{T}^\infty)}|}{\|f\|_{H^p(\mathbb{T}^\infty)}} \\ \leq \sup_{f \in \operatorname{Lin}(\mathbb{T}^\infty)} \frac{\|f\|_{H^2(\mathbb{T}^\infty)}\|\varphi\|_{H^2(\mathbb{T}^\infty)}}{\|f\|_{H^p(\mathbb{T}^\infty)}} \leq \frac{{\|\varphi\|_{H^2(\mathbb{T}^\infty)}}}{a_p}.
	\end{multline*}
	Conversely, Khintchine's inequality \eqref{eq:khintchine} also gives that
	\[\|\varphi\|_{(H^p(\mathbb{T}^\infty))^\ast} = \sup_{f \in H^p(\mathbb{T}^\infty)} \frac{|\langle f, \varphi \rangle_{H^2(\mathbb{T}^\infty)}|}{\|f\|_{H^p(\mathbb{T}^\infty)}} \geq \frac{\|\varphi\|_{H^2(\mathbb{T}^\infty)}^2}{\|\varphi\|_{H^p(\mathbb{T}^\infty)}} \geq \frac{\|\varphi\|_{H^2(\mathbb{T}^\infty)}}{b_p},\]
	since $\varphi$ is in $H^p(\mathbb{T}^\infty)$. To prove optimality of the constants, we appeal to Lemma~\ref{lem:dualinverse} and consider $\varphi_d(z) = (z_1+\cdots+z_d)/\sqrt{d}$ for $d=1$ and as $d\to\infty$.
\end{proof}

Theorem~\ref{thm:general} also follows easily from Lemma~\ref{lem:dualinverse} and \eqref{eq:clt}. 
\begin{proof}
	[Proof of Theorem~\ref{thm:general}] Let $2 \leq p \leq q \leq \infty$ and set $q^{-1}+r^{-1}=1$. Suppose that 
	\begin{equation}\label{eq:gammagamma} 
		{\Gamma\left(1+\frac{p}{2}\right)}^\frac{1}{p} {\Gamma\left(1+\frac{r}{2}\right)}^\frac{1}{r}>1. 
	\end{equation}
	We want to to prove that the Riesz projection is unbounded from $L^q(\mathbb{T}^\infty)$ to $L^p(\mathbb{T}^\infty)$. In view of \eqref{eq:dichotomy}, it is sufficient to find $\psi$ in $L^q(\mathbb{T}^\infty)$ such that
	\[\frac{\|P\psi\|_{L^p(\mathbb{T}^\infty)}}{\|\psi\|_{L^q(\mathbb{T}^\infty)}}>1.\]
	We pick $\psi_d$ in $L^q(\mathbb{T}^\infty)$ of minimal norm such that $P\psi_d = \varphi_d$, where $\varphi_d$ denotes the function from Lemma~\ref{lem:dualinverse}. By \eqref{eq:hahnbanach} and Lemma~\ref{lem:dualinverse}, we obtain
	\[\frac{\|P\psi_d\|_{L^p(\mathbb{T}^\infty)}}{\|\psi_d\|_{L^q(\mathbb{T}^\infty)}} = \|\varphi_d\|_{L^p(\mathbb{T}^\infty)}\|\varphi_d\|_{L^r(\mathbb{T}^\infty)}.\]
	By \eqref{eq:clt} and our assumption \eqref{eq:gammagamma}, the right hand side is strictly larger than $1$ for some sufficiently large $d$. 
\end{proof}

\section{Minimal $L^q(\mathbb{T}^\infty)$ norm} \label{sec:minimal} 
We will now solve the following problem: For $1<q\leq \infty$, find the element $\psi$ in $L^q(\mathbb{T}^\infty)$ of minimal norm such that
\[P\psi(z) = z_1+z_2+\cdots+z_d = \varphi(z).\]
The strict convexity of $L^q(\mathbb{T}^\infty)$ when $1<q<\infty$ means that the minimizer is unique. Uniqueness of the minimizer holds also for $q=\infty$, but in this case it is a consequence of the continuity of $\varphi$ on the polytorus (see e.g.~\cite[Sec.~8.2]{Duren}). 

In view of \eqref{eq:hahnbanach} and (the proof of) Lemma~\ref{lem:dualinverse}, we know that $\psi$ satisfies
\begin{equation} \label{eq:holder}
	\|\psi\|_{L^q(\mathbb{T}^\infty)} = \frac{\langle \varphi, \psi \rangle_{L^2(\mathbb{T}^\infty)}}{\|\varphi\|_{L^p(\mathbb{T}^\infty)}} = \frac{d}{\|\varphi\|_{L^p(\mathbb{T}^\infty)}}
\end{equation}
with $p^{-1}+q^{-1}=1$. On the left hand side of \eqref{eq:holder} we have attained equality in H\"older's inequality, which implies that $|\psi| = C |\varphi|^{p-1}$ almost everywhere. Inserting this into the norm expression $\|\psi\|_{L^q(\mathbb{T}^\infty)}$ in \eqref{eq:holder} and using that $(p-1)q=p$, we find that $C = d \|\varphi\|_{L^p(\mathbb{T}^\infty)}^{-p}$. From H\"older's inequality and \eqref{eq:holder} we also see that
\[\langle |\varphi|,|\psi| \rangle_{L^2(\mathbb{T}^\infty)} \leq \|\varphi\|_{L^p(\mathbb{T}^\infty)}\|\psi\|_{L^q(\mathbb{T}^\infty)} = \langle \varphi, \psi \rangle_{L^2(\mathbb{T}^\infty)},\]
which is only possible if $\varphi \overline{\psi}\geq0$ almost everywhere. Combining these observations yields that
\[\psi(z) = \frac{d}{\|\varphi\|_{L^p(\mathbb{T}^\infty)}^p} |\varphi(z)|^{p-2} \varphi(z)\]
is the element in $L^q(\mathbb{T}^\infty)$ of minimal norm such that $P \psi(z) = z_1+z_2+\cdots+z_d = \varphi(z)$ for $1<p\leq \infty$ and $p^{-1}+q^{-1}=1$. Note that $\psi$ is $1$-homogeneous, which we knew in advance by Lemma~\ref{lem:projection}. We can also directly verify that
\[\int_{\mathbb{T}^\infty} \psi(z) \,\overline{z_j}\,dm_\infty(z) = \int_{\mathbb{T}^\infty} \psi(z) \,\frac{\overline{z_1}+\overline{z_2}+\cdots+\overline{z_d}}{d}\,d\mu_\infty(z)=1,\]
since $\psi$ inherits the symmetry of $\varphi$.

When $d=2$, we can actually compute the Fourier series explicitly. We begin by using the trick $z_1+z_2 = z_2(1+z_1\overline{z_2})$ to write $\psi(z) = z_2 \Psi(z_1\overline{z_2})$, where
\[\Psi(z) = \frac{2}{\|1+z\|_{L^p(\mathbb{T})}^p}|1+z|^{p-2}(1+z).\]
Then we get that
\[\frac{\|1+z\|_{L^p(\mathbb{T})}^p}{2} = \frac{1}{2}\int_{-\pi}^\pi |1+e^{i\theta}|^p \,\frac{d\theta}{2\pi} = 2^{p-1} \int_{-\pi}^\pi \cos^p\left(\frac{\theta}{2}\right) \frac{d\theta}{2\pi} = \frac{2^p}{\pi} \int_{0}^{\pi/2} \cos^p(\vartheta)\,d\vartheta.\]
Similarly, we compute: 
\begin{align*}
	\int_{-\pi}^\pi |1+e^{i\theta}|^{p-2}(1+e^{i\theta})\, e^{-ik\theta}\,\frac{d\theta}{2\pi} &= 2^{p-1} \int_{-\pi}^\pi \cos^{p-1}\left(\frac{\theta}{2}\right) \, e^{-i(k-1/2)\theta}\,\frac{d\theta}{2\pi} \\
	&= 2^{p-1} \int_{-\pi/2}^{\pi/2} \cos^{p-1}(\vartheta) \, e^{-i(2k-1)\vartheta}\,\frac{d\vartheta}{\pi} \\
	&= \frac{2^p}{\pi} \int_0^{\pi/2} \cos^{p-1}(\vartheta) \cos((1-2k)\vartheta)\,d\vartheta 
\end{align*}
The latter integral, which contains the former as the special case $k=0,1$ is known (see e.g. \cite[p.~399]{GR}) and we obtain that
\[\int_{-\pi}^\pi |1+e^{i\theta}|^{p-2}(1+e^{i\theta})\, e^{-ik\theta}\,\frac{d\theta}{2\pi} = \frac{1}{p \operatorname{Beta}\left(\frac{p+1-2k+1}{2},\frac{p-1+2k+1}{2}\right)}\]
for $\operatorname{Beta}(x,y) = \Gamma(x)\Gamma(y)/\Gamma(x+y)$. Combining everything, we find that
\[\psi(e^{i\theta_1},e^{i\theta_2}) = \sum_{k\in \mathbb{Z}} \frac{\Gamma(1+p/2)\Gamma(p/2)}{\Gamma(1+p/2-k)\Gamma(p/2+k)}\, e^{ik\theta_1} e^{i(1-k)\theta_2}.\]

\section*{Acknowledgements} The author would like to extend his gratitude to A.~Bondarenko, H.~Hedenmalm, E.~Saksman and K.~Seip for an interesting discussion which culminated in the material presented in Section~\ref{sec:minimal} and to the referee for a helpful suggestion.

\bibliographystyle{amsplain} 
\bibliography{linear} 
\end{document}